\documentclass[12pt]{amsart}

\usepackage{mathpazo}

\theoremstyle{plain}
\newtheorem{theorem}{Theorem}

\newtheorem{lemma}[theorem]{Lemma}
\newtheorem{corollary}[theorem]{Corollary}

\theoremstyle{definition}

\theoremstyle{remark}

\makeatletter
\newcommand{\opnorm}{\@ifstar\@opnorms\@opnorm}
\newcommand{\@opnorms}[1]{%
  \left|\mkern-1.5mu\left|\mkern-1.5mu\left|
   #1
  \right|\mkern-1.5mu\right|\mkern-1.5mu\right|
}
\newcommand{\@opnorm}[2][]{%
  \mathopen{#1|\mkern-1.5mu#1|\mkern-1.5mu#1|}
  #2
  \mathclose{#1|\mkern-1.5mu#1|\mkern-1.5mu#1|}
}
\makeatother

\begin{document}

\title[Amalgamations of Banach spaces]{Zippin's embedding theorem and amalgamations of classes of Banach spaces}
\author{Ond\v{r}ej Kurka}
\thanks{The research was supported by the grant GA\v{C}R 14-04892P. The author is a junior researcher in the University Centre for Mathematical Modelling, Applied Analysis and Computational Mathematics (MathMAC). The author is a member of the Ne\v{c}as Center for Mathematical Modeling.}
\address{Department of Mathematical Analysis, Charles University, Soko\-lovsk\'a 83, 186 75 Prague 8, Czech Republic}
\email{kurka.ondrej@seznam.cz}
\begin{abstract}
It was proved by Dodos and Ferenczi that the classes of Banach spaces with a separable dual and of separable reflexive Banach spaces are strongly bounded. In this note, we provide an isometric version of this result.
\end{abstract}
\maketitle

A class $ \mathcal{P} $ of separable Banach spaces is called \emph{strongly bounded} if, for every analytic subset $ \mathcal{C} $ of $ \mathcal{P} $ in the Effros-Borel structure of subspaces of $ C(2^{\mathbb{N}}) $ (defined below), there exists $ E \in \mathcal{P} $ that contains an isomorphic copy of every $ X \in \mathcal{C} $. This concept was introduced by S.~A.~Argyros and P.~Dodos \cite{argyrosdodos} who proved that several classes, namely the class of spaces with a shrinking basis and the class of reflexive spaces with a basis, are strongly bounded.

Let $ X $ be a Banach space with a separable dual. Zippin's theorem \cite{zippin} states that $ X $ embeds isomorphically into a Banach space with a shrinking basis. Moreover, if $ X $ is reflexive, then it embeds into a reflexive Banach space with a basis. The theorem was revisited by N.~Ghoussoub, B.~Maurey and W.~Schachermayer \cite{gms}. Following their approach and an ordinal version by B.~Bossard \cite{bossard1}, the reliance on a basis was dropped by P.~Dodos and V.~Ferenczi \cite{dodosferenczi} who proved that the classes of Banach spaces with a separable dual and of separable reflexive Banach spaces are strongly bounded.

G.~Godefroy asked in \cite{godefroy} for an isometric version of the amalgamation theory of S.~A.~Argyros and P.~Dodos \cite{argyrosdodos}. Recently, the author \cite{kurka} has introduced an isometric variant of their technique. However, the technique from \cite{kurka} is applicable only on spaces with a monotone basis, so the problem of the reliance on a basis appears again. For the two considered classes, it turns out that the problem can be solved by the same method as in the isomorphic setting. We obtain the following extension of the result from \cite{dodosferenczi}.

\begin{theorem} \label{thmmain}
{\rm (1)} Let $ \mathcal{C} $ be an analytic set of Banach spaces with a separable dual. Then there exists a Banach space $ E $ with a shrinking monotone basis which contains an isometric copy of every member of $ \mathcal{C} $.

{\rm (2)} Let $ \mathcal{C} $ be an analytic set of separable reflexive Banach spaces. Then there exists a reflexive Banach space $ E $ with a monotone basis which contains an isometric copy of every member of $ \mathcal{C} $.
\end{theorem}

We include two consequences of this theorem which are nothing else than isometric versions of two corollaries from \cite{dodosferenczi}. The first of them provides a universal space for a class of spaces on which the Szlenk index is bounded.

\begin{corollary} \label{corszlenk}
For every $ \xi < \omega_{1} $, there exists a Banach space $ E_{\xi} $ with a separable dual which contains an isometric copy of every separable Banach space with the Szlenk index less than or equal to $ \xi $.
\end{corollary}

E.~Odell and Th.~Schlumprecht \cite{odschl} proved, answering a question of J.~Bourgain \cite{bourgain}, that there exists a separable reflexive Banach space which contains an isomorphic copy of every separable super-reflexive Banach space. (The same result was proposed also in \cite{tokarev}). We show that this result holds in a stronger sense.

\begin{corollary} \label{corsuprefl}
There exists a separable reflexive Banach space which contains an isometric copy of every separable super-reflexive Banach space.
\end{corollary}

In particular, the space contains an isometric copy of every finite-dimensional space. Let us note that a result of A.~Szankowski \cite{szankowski} (cf.~\cite[remark (II)]{kurka}) states that there exists a separable reflexive Banach space in which every finite-dimensional space has an isometric copy that is $ 1 $-complemented.

Let us introduce some definitions and notation. By a basis we mean a Schauder basis. A basis $ x_{1}, x_{2}, \dots $ is said to be \emph{monotone} if the associated partial sum operators $ P_{n} : \sum_{k=1}^{\infty} a_{k}x_{k} \mapsto \sum_{k=1}^{n} a_{k}x_{k} $ satisfy $ \Vert P_{n} \Vert \leq 1 $.

A \emph{Polish space (topology)} means a separable completely metrizable space (topology). A set $ P $ equipped with a $ \sigma $-algebra is called a \emph{standard Borel space} if the $ \sigma $-algebra is generated by a Polish topology on $ P $. A subset of a standard Borel space is called \emph{analytic} if it is a Borel image of a Polish space. The complement of an analytic set is called \emph{coanalytic}.

For a topological space $ X $, the set $ \mathcal{F}(X) $ of all closed subsets of $ X $ is equipped with the \emph{Effros-Borel structure}, defined as the $ \sigma $-algebra generated by the sets
$$ \{ F \in \mathcal{F}(X) : F \cap U \neq \emptyset \} $$
where $ U $ varies over open subsets of $ X $. If $ X $ is Polish, then, equipped with this $ \sigma $-algebra, $ \mathcal{F}(X) $ forms a standard Borel space.

By the \emph{standard Borel space of separable Banach spaces} we mean
$$ \mathcal{SE}(C(2^{\mathbb{N}})) = \big\{ F \in \mathcal{F}(C(2^{\mathbb{N}})) : \textrm{$ F $ is linear} \big\}, $$
considered as a subspace of $ \mathcal{F}(C(2^{\mathbb{N}})) $.

The interpolation scheme of Davis, Figiel, Johnson and Pe\l czy\'nski \cite{dfjp} is defined as follows. Let $ W $ be a bounded, closed, convex and symmetric subset of a Banach space $ X $. For each $ n \in \mathbb{N} $, let $ \Vert \cdot \Vert_{n} $ be the equivalent norm given by
$$ B_{(X, \Vert \cdot \Vert_{n})} = \overline{2^{n}W + 2^{-n}B_{X}}. $$
The \emph{$ 2 $-interpolation space of the pair $ (X, W) $} is defined as the space $ (Y, \opnorm{\cdot}) $ where
$$ \opnorm{x} = \Big( \sum_{n=1}^{\infty} \Vert x \Vert_{n}^{2} \Big)^{1/2}, \quad x \in X, $$
and
$$ Y = \{ x \in X : \opnorm{x} < \infty \}. $$

Now, let us prove the above results. Actually, the only remaining ingredient is the following observation.

\begin{lemma} \label{interpembed}
Let $ E $ be a subspace of $ X $ such that $ B_{E} \subset W \subset B_{X} $. Then there is a constant $ c > 0 $ such that
$$ \opnorm{x} = c \Vert x \Vert, \quad x \in E. $$
In particular, $ Y $ contains an isometric copy of $ E $.
\end{lemma}

\begin{proof}
For every $ n \in \mathbb{N} $, our assumption provides
$$ 2^{n}B_{E} + 2^{-n}B_{E} \subset 2^{n}W + 2^{-n}B_{X} \subset 2^{n}B_{X} + 2^{-n}B_{X}, $$
and so
$$ (2^{n}+2^{-n}) B_{E} \subset B_{(X, \Vert \cdot \Vert_{n})} \subset (2^{n} + 2^{-n}) B_{X}. $$
We obtain
$$ \Vert x \Vert_{n} = \frac{1}{2^{n}+2^{-n}} \Vert x \Vert, \quad x \in E, $$
and the existence of a suitable constant $ c $ follows.
\end{proof}

\begin{proof}[Proof of Theorem \ref{thmmain}]
By \cite[Theorem 1.2]{kurka}, the theorem has been already proven under the assumption that the members of $ \mathcal{C} $ have a shrinking monotone basis (considering part (1)) or a monotone basis (considering part (2)). Moreover, using the Lusin separation theorem (see e.g. \cite[(14.7)]{kechris}), since both the families of spaces with a separable dual and of reflexive spaces are coanalytic (see \cite[Corollary 3.3]{bossard2}), we may assume that $ \mathcal{C} $ is Borel. Therefore, it is sufficient to show the following.

(1) \emph{Let $ \mathcal{B} $ be a Borel set of Banach spaces with a separable dual. Then there exists an analytic set $ \mathcal{C} $ of Banach spaces with a shrinking monotone basis such that an isometric copy of every member of $ \mathcal{B} $ is contained in a member of $ \mathcal{C} $.}

(2) \emph{Let $ \mathcal{B} $ be a Borel set of separable reflexive Banach spaces. Then there exists an analytic set $ \mathcal{C} $ of reflexive Banach spaces with a monotone basis such that an isometric copy of every member of $ \mathcal{B} $ is contained in a member of $ \mathcal{C} $.}

We show that it is possible to use a method of Ghoussoub, Maurey and Schachermayer \cite{gms}, studied further in \cite{bossard1} and \cite[Chapter~5]{dodostopics}. Let $ \mathbf{1} $ denote the constant function on $ 2^{\mathbb{N}} $ equal to $ 1 $ and let $ g_{0} \in C(2^{\mathbb{N}}) $ be a fixed function that separates points of $ 2^{\mathbb{N}} $. For every subspace $ E \subset C(2^{\mathbb{N}}) $ containing $ \mathbf{1} $ and $ g_{0} $ and with $ E^{*} $ separable, it is possible to find a monotone basis $ e^{E}_{1}, e^{E}_{2}, \dots $ of $ C(2^{\mathbb{N}}) $ with the associated partial sum operators $ P^{E}_{1}, P^{E}_{2}, \dots $ such that, if we define $ W_{E} $ by
$$ W_{E} = \overline{\mathrm{co}} \, \bigcup_{k=1}^{\infty} P^{E}_{k}(B_{E}) $$
and $ Z(E) $ as the $ 2 $-interpolation space of the pair $ (C(2^{\mathbb{N}}), W_{E}) $, then the following properties are satisfied.

(i) $ Z(E) $ contains an isomorphic copy of $ E $ (see \cite{gms} or \cite[Theorem~5.17]{dodostopics}).

(ii) The sequence $ e^{E}_{1}, e^{E}_{2}, \dots $ is contained in $ Z(E) $ and forms a basis of $ Z(E) $ that is shrinking and monotone (see \cite{gms} or \cite[Theorem~5.17]{dodostopics}).

(iii) If $ E $ is reflexive, then $ Z(E) $ is reflexive (see \cite[Lemma~5.18]{dodostopics}).

(iv) If we denote
$$ E_{X} = \overline{\mathrm{span}} \, \big( X \cup \{ \mathbf{1}, g_{0} \} \big), $$
then, for every Borel set $ \mathcal{B} $ of Banach spaces with a separable dual, the set
$$ \mathcal{Z} = \Big\{ (X, Y) \in \mathcal{B} \times \mathcal{SE}(C(2^{\mathbb{N}})) : \textrm{$ Y $ is isometric to $ Z(E_{X}) $} \Big\} $$
is analytic (see \cite{bossard1} or \cite[Theorem~5.19]{dodostopics}).

Now, we realize that (i) holds in a stronger sense. Using Lemma~\ref{interpembed}, since $ B_{E} \subset W_{E} \subset B_{C(2^{\mathbb{N}})} $, we obtain that $ Z(E) $ contains an isometric copy of $ E $. Thus, $ Z(E_{X}) $ contains an isometric copy of $ X $.

Let $ \mathcal{B} $ be a Borel set of Banach spaces with a separable dual. To find a suitable $ \mathcal{C} $, we consider the projection of $ \mathcal{Z} $ on the second coordinate. By (iv), the set
$$ \mathcal{C} = \Big\{ Y \in \mathcal{SE}(C(2^{\mathbb{N}})) : \textrm{$ Y $ is isometric to $ Z(E_{X}) $ for some $ X \in \mathcal{B} $} \Big\} $$
is analytic. It follows from properties (ii), (iii) and the isometric version of (i) that $ \mathcal{C} $ works.
\end{proof}

\begin{proof}[Proof of Corollary \ref{corszlenk}]
It follows from \cite[Theorem~4.11]{bossard2} and \cite[Proposition~0.1(i)]{bossard2} that the set of spaces $ X \subset C(2^{\mathbb{N}}) $ with $ \mathrm{Sz}(X) \leq \xi $ is Borel. In particular, it is analytic, and Theorem~\ref{thmmain}(1) can be applied.
\end{proof}

\begin{proof}[Proof of Corollary \ref{corsuprefl}]
The set of uniformly convex subspaces of $ C(2^{\mathbb{N}}) $ is Borel (see \cite[Corollary~5]{dodosferenczi}). Since the relation of isomorphism is analytic in $ \mathcal{SE}(C(2^{\mathbb{N}}))^{2} $ (see \cite[Theorem~2.3(i)]{bossard2}), the set of subspaces of $ C(2^{\mathbb{N}}) $ which have an equivalent uniformly convex norm is analytic. This set coincides with the set of super-reflexive subspaces (see \cite[Corollary~3]{enflo}). Thus, Theorem~\ref{thmmain}(2) can be applied.
\end{proof}

\end{document}